\documentclass[
framed_theorems,
framed_proofs,
print,
]{OCPreprint}

\usepackage{enumitem}

\renewcommand\d{\mathop{}\!\mathrm{d}}

\title[Non-existence of perturbed solutions under a SSC]
{Non-existence of perturbed solutions under a second-order sufficient condition}
\author{Gerd Wachsmuth\footnote{%
		Brandenburgische Technische Universität Cottbus--Senftenberg,
		Institute of Mathematics,
		03046 Cottbus,
		Germany,
		\email{wachsmuth@b-tu.de},
		\url{https://www.b-tu.de/fg-optimale-steuerung/team/prof-gerd-wachsmuth},
		ORCID: 0000-0002-3098-1503%
	}
}
{
	\makeatletter
	\def\and{ and }
	\def\footnote#1{}
	\hypersetup{
		pdftitle={\@title},
		pdfauthor={\@author}
	}
	\makeatother
}

\begin{document}
\maketitle
\begin{abstract}
	We present an optimization problem in infinite dimensions
	which satisfies the usual second-order sufficient condition
	but for which perturbed problems fail to possess solutions.
\end{abstract}


\begin{msc}
	\mscLink{49K40}
\end{msc}

\section{Introduction}
For nonlinear optimization problems in finite dimensions,
it is well known
that uniqueness of the Lagrange multiplier
together with the standard second-order sufficient condition (SSC)
leads to the stability of the solution
under perturbations of the problem
and this is a key result
for the analysis of SQP methods,
see
\cite[Proposition~6.3]{Bonnans1994}
and
\cite[Theorem~4.20]{Robinson1982}.
For infinite-dimensional problems,
such a result is not known
and one typically resorts
to a strong second-order sufficient condition
for the analysis of SQP-type methods,
see, e.g., \cite[(2.17)]{Malanowski1992} and 
\cite[(2.3)]{Alt1990}.
In this note,
we present an example which satisfies
the usual second-order sufficient conditions
but for which perturbed problems
fail to possess solutions.

\section{Second-order sufficient conditions and perturbations}
For convenience,
we state the standard SSC
for a problem of type
\begin{equation*}
	\text{Minimize}
	\quad
	f(x)
	\quad\text{s.t.\ } x \in C,
\end{equation*}
where
$X$ is a Banach space,
$f \colon X \to \R$
is twice continuously Fréchet differentiable on $X$
and
$C \subset X$ is closed and convex.
Due to the simple structure of the constraint,
Robinson's constraint qualification
is trivially satisfied by the above problem
and we have the following
result,
see \cite[Theorem~3.63(i), Remark~3.68]{BonnansShapiro2000}.
\begin{theorem}
	\label{thm:SSC}
	Let $\bar x \in C$, $\beta,\eta > 0$ be given,
	such that
	the conditions
	\begin{subequations}
		\label{eq:SSC}
		\begin{align}
			\label{eq:SSC:1}
			-f'(\bar x) &\in \NN_C(\bar x)
			\\
			\label{eq:SSC:2}
			f''(\bar x) h^2
			&\ge
			\beta \norm{h}_X^2
			\qquad
			\forall
			h \in \hat C_\eta(\bar x)
		\end{align}
	\end{subequations}
	hold.
	Then, a quadratic growth condition is valid at $\bar x$,
	i.e., there exist
	$\delta,\varepsilon > 0$
	such that
	\begin{equation}
		\label{eq:growth}
		f(x) \ge f(\bar x) + \frac{\delta}{2} \norm{x - \bar x}_X^2
		\qquad
		\forall x \in C, \norm{x - \bar x}_X \le \varepsilon.
	\end{equation}
\end{theorem}
Here, $\TT_C(\bar x)$ and $\NN_C(\bar x)$ are
the tangent cone
and the normal cone in the sense of convex analysis,
respectively,
and
\begin{equation*}
	\hat C_\eta(\bar x)
	:=
	\set{
		h \in \TT_C(\bar x)
		\given
		f'(\bar x) h \le \eta \norm{h}_X
	}
\end{equation*}
is the so-called approximate critical cone.
The role of the Lagrange multiplier
is played by $-f'(\bar x)$
and this element is uniquely determined (for fixed $\bar x$).

In finite dimensions,
the assertion of \cref{thm:SSC}
also gives stability w.r.t.\ perturbations.
As an example, we give a simple result with linear perturbations.
\begin{proposition}
	\label{prop:perturbation}
	Suppose that $X$ is finite dimensional
	and that
	$\bar x \in C$, $\delta,\varepsilon > 0$
	satisfy \eqref{eq:growth}.
	Then, for any $x\dualspace \in X\dualspace$
	with
	$\norm{x\dualspace}_{X\dualspace} < \delta \varepsilon / 2$
	the perturbed problem
	\begin{equation*}
		\text{Minimize}
		\quad
		f(x) + \dual{x\dualspace}{x}_X
		\quad\text{s.t.\ } x \in C
	\end{equation*}
	has a local solution $z$ satisfying
	$\norm{z - \bar x}_X \le 2 \delta^{-1} \norm{x\dualspace}_{X\dualspace}$.
\end{proposition}
\begin{proof}
	We consider the localized problem
	\begin{equation*}
		\text{Minimize}
		\quad
		f(x) + \dual{x\dualspace}{x}_X
		\quad\text{s.t.\ } x \in C \cap B_\varepsilon(\bar x).
	\end{equation*}
	Since the objective is continuous
	and since the feasible set is compact,
	there exists a local solution $z$.
	The optimality of $z$ together with the growth condition \eqref{eq:growth}
	implies
	\begin{align*}
		\frac\delta2 \norm{z - \bar x}_X^2
		&\le
		f(z) - f(\bar x)
		=
		\parens*{
			f(z) + \dual{x\dualspace}{z}_X
		}
		-
		\parens*{
			f(\bar x) + \dual{x\dualspace}{\bar x}_X
		}
		+
		\dual{x\dualspace}{ \bar x - z }_X
		\\&\le
		\norm{x\dualspace}_{X\dualspace}
		\norm{z - \bar x}_X
		.
	\end{align*}
	Thus,
	$\norm{ z - \bar x }_X \le 2 \delta^{-1} \norm{x\dualspace}_{X\dualspace} < \varepsilon$.
	This shows the desired inequality
	and since the constraint $z \in B_\varepsilon(\bar x)$ is not binding,
	$z$ is a local solution of the perturbed problem.
\end{proof}

\section{Counterexample}
We present a counterexample which shows that
\cref{prop:perturbation}
fails in infinite dimensions.
We define the operator $S \colon L^2(0,1) \to L^2(0,1)$
via
$(S u)(x) = \int_0^x u(\xi) \d\xi$.
We further define
\begin{align*}
	X &:= \R \times L^2(0,1),
	\qquad
	C := \set[\big]{ (t,u) \in X \given \abs{u} \le t \text{ a.e. on } (0,1) }, \\
	f_h(t,u) &:= t^2 + \norm{S u}_{L^2(0,1)}^2 - \frac12 \norm{u}_{L^2(0,1)}^2 - h t
	\qquad \forall (t,u) \in X
	,
\end{align*}
where $h \ge 0$ is a linear perturbation parameter
and $h = 0$ corresponds to the unperturbed problem.
It is clear that
$X$ is a Banach space, $C \subset X$ is closed and convex
and $f_h$ is twice Fréchet differentiable with
\begin{equation*}
	\dual{f_h'(\bar t, \bar u)}{(t,u)}_X
	=
	2 \bar t t + 2 \dual{S \bar u}{S u}_{L^2(0,1)} - \dual{\bar u}{u}_{L^2(0,1)} - h t
\end{equation*}
and
\begin{equation*}
	f_h''(\bar t, \bar u) (t, u)^2
	=
	2 t^2 + 2 \norm{S u}_{L^2(0,1)}^2 - \norm{u}_{L^2(0,1)}^2
	.
\end{equation*}

We start by checking that
the point
$(\bar t, \bar u) = (0,0)$
satisfies the SSC from \cref{thm:SSC} for the unperturbed objective $f_0$.
It is clear that $(0,0) \in C$
and $f_0'(0,0) = 0 \in \NN_C(0,0)$.
Since $C$ is a closed, convex cone, we have
$\TT_C(0,0) = C$,
thus,
the approximate critical cone is given by
$\hat C_\eta(0,0) = C$.
It remains to verify \eqref{eq:SSC:2}.
Let $(t,u) \in C$ be arbitrary.
Then, $\abs{u} \le t$ implies $\abs{ (S u)(x) } \le t x$.
Thus,
\begin{align*}
	f_0''(0,0)
	(t,u)^2
	&=
	2 t^2 + 2 \norm{S u}_{L^2(0,1)}^2 - \norm{u}_{L^2(0,1)}^2
	\\&
	\ge
	2 t^2 - 2 \int_0^1 (t x)^2 \d x - \int_0^1 t^2 \d x
	=
	\parens*{ 2 - 2/3 - 1 } t^2
	=
	\frac13 t^2
	\\&\ge
	\frac16 \parens*{t^2 + \norm{u}_{L^2(0,1)}^2 }
	=
	\frac16 \norm{(t,u)}_X^2
\end{align*}
and this verifies \eqref{eq:SSC:2} with $\beta = 1/6$.
Thus, $(\bar t, \bar u) = (0,0)$
is a local minimizer of $f_0$ on $C$
and one can show that it is even a global minimizer on $C$.

It remains to compute the minimizers of $f_h$ on $C$.
Thus,
let $h \ge 0$ be given and let $(\tilde t, \tilde u) \in C$
be a local minimizer of $f_h$ on $C$.

First, we consider the case $\tilde t > 0$.
Then, $\tilde u$ is a minimizer of
\begin{equation*}
	\text{Minimize} 
	\quad
	\norm{S u}_{L^2(0,1)}^2 - \frac12 \norm{u}_{L^2(0,1)}^2
	\quad\text{s.t.\ }
	-\tilde t \le u \le \tilde t
	\text{ a.e.\ on $(0,1)$}.
\end{equation*}
Now, Pontryagin's maximum principle implies that
$\tilde u(x)$ solves
\begin{equation*}
	\text{Minimize}
	\quad
2 (S\adjoint S \tilde u)(x) u - \frac12 u^2
\quad
\text{s.t.\ } u \in [-\tilde t, \tilde t]
\end{equation*}
for almost every $x \in (0,1)$.
Here, $S\adjoint S \tilde u$ is given by
$(S\adjoint S\tilde u)(x) = \int_x^1 (S \tilde u)(\xi) \d\xi $.
In particular, we get the implications
\begin{align*}
	(S\adjoint S\tilde u)(x) &> 0
	\quad\Rightarrow\quad
	\tilde u(x) = -\tilde t
	\\
	(S\adjoint S\tilde u)(x) &< 0
	\quad\Rightarrow\quad
	\tilde u(x) = \tilde t.
\end{align*}
We argue that $S\adjoint S \tilde u = 0$.
Let $x \in [0,1)$ be a maximizer of the continuous function $S\adjoint S \tilde u$.
Towards a contradiction, suppose that
$(S\adjoint S \tilde u)(x) > 0$.
Then, we get $\tilde u = -\tilde t$ in a neighborhood of $x$
and this implies
\begin{equation*}
	(S\adjoint S \tilde u)'(\xi) = -(S \tilde u)(\xi)
	\quad\text{and}\quad
	(S\adjoint S \tilde u)''(\xi) = -\tilde u(\xi) = \tilde t > 0.
\end{equation*}
for all $\xi$ in a neighborhood of $x$.
Together with $(S \tilde u)(0) = 0$,
this contradicts the assumption that $x$ is a maximizer of $S\adjoint S \tilde u$.
Hence,
$S\adjoint S \tilde u \le 0$.
Similarly, we can show
$S\adjoint S \tilde u \ge 0$.
Thus,
$S\adjoint S \tilde u = 0$
and this gives
$\tilde u = 0$
which is, obviously, not a minimizer.
Hence, every local minimizer $(\tilde t, \tilde u)$ satisfies $\tilde t = 0$
and, consequently, $\tilde u = 0$.
This shows that $\tilde t = 0$ is a local minimizer of
\begin{equation*}
	\text{Minimize}
	\quad
	t^2 - h t
	\quad\text{s.t.\ }
	t \ge 0.
\end{equation*}
This implies $h = 0$.

To summarize,
$(0,0)$ is the only (local) minimizer of $f_0$ on $C$
and satisfies the second-order sufficient conditions.
However, the linearly perturbed functionals
$f_h$, $h > 0$, do not admit local minimizers on $C$.


\renewcommand*{\bibfont}{\small}
\printbibliography

\end{document}